\documentclass[article]{amsart}
\usepackage{helvet, color}
\usepackage{commath, comment}
\usepackage{xargs, soul}        
\usepackage[colorinlistoftodos,prependcaption,textsize=tiny]{todonotes}
\newcommandx{\deb}[2][1=]{\todo[linecolor=purple,backgroundcolor=purple!25,bordercolor=purple,#1]{#2}}
\newcommandx{\florian}[2][1=]{\todo[linecolor=red,backgroundcolor=red!25,bordercolor=red,#1]{#2}}
 \usepackage{marginnote}
\usepackage{amscd,amsmath,amsxtra,amsthm,amssymb,stmaryrd,xr,mathrsfs,mathtools,enumerate,commath}
\usepackage{hyperref}
\hypersetup{
 colorlinks=true,
 linkcolor=teal,
 filecolor=magenta,
 citecolor=orange,
 urlcolor=violet,
 pdftitle={Statistics of Anti-cyclotomic Invariants at ss primes},
 pdfpagemode=FullScreen,
 }
\usepackage{tikz-cd}
\usepackage[all,cmtip]{xy}
\usepackage[margin=1.2 5in]{geometry}
\DeclareFontFamily{U}{wncy}{}
\DeclareFontShape{U}{wncy}{m}{n}{<->wncyr10}{}
\DeclareSymbolFont{mcy}{U}{wncy}{m}{n}
\DeclareMathSymbol{\Sha}{\mathord}{mcy}{"58}

\newtheorem{theorem}{Theorem}[section]
\newtheorem*{theorem*}{Theorem}
\newtheorem{lemma}[theorem]{Lemma}

\newtheorem{proposition}[theorem]{Proposition}

\newtheorem{defi}[theorem]{Definition}

\numberwithin{equation}{section}

\theoremstyle{remark}
\newtheorem{remark}[theorem]{Remark}

\newcommand{\tr}{\operatorname{tr}}

\newcommand{\Gal}{\operatorname{Gal}}

\newcommand{\ac}{\operatorname{ac}}

\newcommand{\Aut}{\operatorname{Aut}}

\newcommand{\Qp}{\mathbb{Q}_p}
\newcommand{\Zp}{\mathbb{Z}_p}

\newcommand{\GL}{\mathrm{GL}}

\newcommand*{\EC}{\mathsf{E}}

\newcommand{\Z}{\mathbb{Z}}

\newcommand{\Q}{\mathbb{Q}}
\newcommand{\F}{\mathbb{F}}

\newcommand{\Hom}{\mathrm{Hom}}

\newcommand{\Sel}{\mathrm{Sel}}

\newcommand{\Dd}{D_{\Q(\sqrt{d})}}

\newcommand{\Selpinf}{\Sel_{p^{\infty}}(E/K_{\ac})}

 \newcommand{\widebar}[1]{\mkern 2.5mu\overline{\mkern-2.5mu#1\mkern-2.5mu}\mkern 2.5mu}
 
 \newcommand{\barrho}{\overline{\rho}}

\makeatletter
\theoremstyle{theorem}
\newtheorem*{c@njecture}{\conjn@name}
\newcommand{\myl@bel}[2]{%
  \protected@write \@auxout {}{\string \newlabel {#1}{{#2}{\thepage}{#2}{#1}{}} }%
 \hypertarget{#1}{}
    } 
\newenvironment{labelledconj}[3][]%
    {
        \def\conjn@name{#2}
        \begin{c@njecture}[{#1}]\myl@bel{#3}{#2}
    }
    {
        \end{c@njecture}
    }
\makeatother

\begin{document}
\title[Cotorsion of Anti-Cyclotomic Selmer groups on average]{Cotorsion of Anti-Cyclotomic Selmer groups on average}

\author[D.~Kundu]{Debanjana Kundu}
\address[Kundu]{Fields Institute \\ 
University of Toronto \\
Toronto ON, M5T 3J1, Canada}
\email{dkundu@math.toronto.edu}

\author[F.~Sprung]{Florian Sprung}
\address[Sprung]{School of Mathematical and Statistical Sciences, Arizona State University\\ Tempe, AZ 85287-1804}
\email{florian.sprung@asu.edu}

\begin{abstract}{For an elliptic curve, we study how many Selmer groups are cotorsion over the anti-cyclotomic $\Zp$-extension as one varies the prime $p$ or the quadratic imaginary field in question.}
\end{abstract}

\subjclass[2010]{11G05, 11R23 (primary); 11R45 (secondary).}
\keywords{Iwasawa theory, Selmer groups, elliptic curves, anti-cyclotomic, supersingular primes}

\maketitle

\section{Introduction}
Let $\EC$ be an elliptic curve defined over $\Q$ and let $\Q(\sqrt{d})$ be an imaginary quadratic field.
Much of the arithmetic of $\Q(\sqrt{d})$-rational points of $\EC$ is contained in the behaviour of points of $\EC$ up the anti-cyclotomic $\Zp$-tower $\Q(\sqrt{d})_{\ac}$ of $\Q(\sqrt{d})$ for a prime $p$ of good reduction.

The goal of this paper is to count the proportion of anti-cyclotomic Selmer groups whose behaviour we understand.
There are two different cases to be considered.
The first is the \emph{indefinite case}, in which the number of bad primes of $\EC$ that are inert in $\Q(\sqrt{d})$ is even.
This condition is also known as the (generalized) Heegner hypothesis, and should allow for many rational points on the elliptic curve.
Indeed, M.~Bertolini \cite{Ber95} and M.~Longo--S.~Vigni \cite{LV} show in various scenarios that the appropriate anti-cyclotomic Selmer groups have corank $1$ over the anti-cyclotomic Iwasawa algebra.

The second one is the \emph{definite case} and was studied by R.~Pollack and T.~Weston in \cite{PW11}, and in their joint work with C.~Kim \cite{KPW17}.
Here, the number of bad inert primes is odd, preventing the existence of Heegner points.
Consequently, there should be few rational points.
Their work confirms this and shows that under various hypotheses, the anti-cyclotomic Selmer group is cotorsion.

While the hypotheses employed in works concerning the first (indefinite) case are mild from a statistical point of view, the ones employed in the second (definite) case (i.e. \cite{BD05, KPW17,PW11}) cut down the proportion of provably corank $0$ Selmer groups, and we are interested in counting what this proportion is.
To this end, there are known results in some cases:
For a fixed pair ($\EC_{/\Q},p$) where $p$ is ordinary, \cite{HKR21} gave a lower bound in this definite case for the proportion of imaginary quadratic fields of the form $\Q(\sqrt{-\ell})$ with ${\ell}$ a prime for which the Selmer groups are known to be cotorsion.
Our paper generalizes \cite{HKR21} in two ways: 
\begin{enumerate}
\item We include all imaginary quadratic fields $\Q(\sqrt{d})$ in the count.
\item We remove the ordinarity hypothesis by including supersingular primes for which $a_p=0$.
\end{enumerate}
The main result of our paper measures from a statistical point of view how mild the assumptions in \cite{KPW17} and \cite{PW11} are.
A bit more precisely, it asserts that the proportion of such imaginary quadratic fields is halved (i.e. multiplied by $\frac{1}{2}$) for each prime of bad reduction that is \textit{split} that would violate the key hypothesis of \cite{KPW17} \textit{were it inert}.
Instead of overwhelming the reader with precise statements, we give a flavour via an example.
(For the technically savvy reader: they are Theorems~\ref{thm: vary K hypothesis CR} and \ref{thm: vary K hypothesis CR ordinary}, addressing the supersingular and the ordinary cases separately.)

The elliptic curve \href{https://www.lmfdb.org/EllipticCurve/Q/497/a/1}{497a1} with Weierstrass equation $y^2+xy=x^3+x^2+25x-14$ has bad reduction at $7$ and $71$.
At the prime $5$, this elliptic curve attains good supersingular reduction.
Our theorem then says that the proportion of cotorsion anti-cyclotomic Selmer groups as one varies the quadratic imaginary field is at least 
\[
\frac{1}{4}\times\frac{5\times7\times71}{6\times 8\times 72}=0.1797598\cdots.
\]
When randomly choosing imaginary quadratic fields, $p$ should split half of the time and we should land in the definite case half of the time.
This accounts for the factor of $\frac{1}{4}=\frac{1}{2}\times\frac{1}{2}$.
The other factor occurs because we count fields with discriminant coprime to $p$ and conductor of the elliptic curve.

If instead of the assumptions in \cite{PW11}, we relied on \cite{BD05}\footnote{\cite[Remark 4.2]{PW11} compares the different assumptions in more detail.}, the lower bound would be $$\frac{1}{8}\times\frac{5\times7\times71}{6\times 8\times 72}=0.0898799\cdots,$$
i.e. our theorem shows that the work of \cite{PW11} doubled the desired proportion!

We achieve the lower bound by counting the proportion of such imaginary quadratic fields that satisfy several hypotheses imposed in \cite{KPW17}, which we call \textit{choired}.
\textit{`Cho'} indicates we are in the definite case, while \textit{`ired'} indicates a ramification hypothesis.
More precisely, \textit{`Choired'} stands for: \textbf{C}onductor shouldn't satisfy  
\textbf{H}eegner hypothesis, so has an
\textbf{O}dd number of factors.
Furthermore,
\textbf{I}nert primes 
\textbf{R}amify only under 
\textbf{E}xtra 
\textbf{D}ifficulty imposed by (Kim--)Pollack--Weston.

Three steps are needed to perform the count.

First, in Lemma~\ref{residueclasses} we encode imaginary quadratic fields with the same splitting type at the bad primes of $\EC$ into something that is easier to count.
We achieve this by working with the discriminants of the fields modulo the conductor of $\EC$ (denoted by $N_{\EC}$) and modulo the prime $p$, showing that each family of such fields corresponds to a proportion of $\frac{1}{2^{r+1}}$ of the possible residue classes, where $r$ is the number of bad primes.

The second step is to estimate the proportion of imaginary quadratic fields with discriminant coprime to $pN_{\EC}$ (Proposition~\ref{lemma: count imaginary quadratic with disc coprime to pNE}).
To do this, we sum appropriate estimates due to K.~Prachar and P.~Humphries over the residue classes in question from the first step.

In the final step, we break up the choired fields into a disjoint union of families of imaginary quadratic fields with prescribed splitting type as in step $1$.
Then using the counting estimates in the previous two steps and a combinatorial count for this disjoint union, we arrive at our estimate.

\emph{Organization}:
Including this introduction, the article has four sections.
Section~\ref{section: preliminaries} is preliminary in nature.
We introduce the definition of the key objects and  also introduce the criterion of Pollack--Weston in this section.
Sections~\ref{varying over primes} and \ref{section: Varying K} are devoted to studying this criterion on average.
Section~\ref{varying over primes} is a warm-up to the main result: we show that for a fixed a pair $(\EC_{/\Q},\Q(\sqrt{d}))$, the Selmer groups are cotorsion for almost all $p$ of good ordinary reduction.
It would be interesting to prove an analogous statement for supersingular primes.
See Theorem~\ref{Th: CR-vary p} for the precise statement.
We then develop the methods to prove our main result in Section \ref{section: Varying K}, following the three steps described above.

\emph{Outlook}: If we fix the pair $\left(p,\Q(\sqrt{d})\right)$ and vary the elliptic curve (ordered by height or conductor) with good reduction at $p$, it seems to be significantly more difficult to estimate for what proportion of elliptic curves the appropriate Selmer groups are $\Lambda$-cotorsion.
We will investigate aspects of this question in future projects.

\section*{Acknowledgements}
We thank Manjul Bhargava, Allysa Lumley, V.~Kumar Murty, Artane Siad, Joseph Silverman, and Andrew Sutherland for stimulating and helpful discussions, Ming-Lun Hsieh, Chan-Ho Kim, Robert Pollack, and Tom Weston for answering our questions related to \cite{KPW17, PW11}, and Noam Elkies for answering a question related to \cite{Elk87}.
We also thank Stefano Vigni for clarifying the papers \cite{Ber95, LV} and Melanie Matchett Wood for answering a question related to counting imaginary quadratic fields.
DK was supported by a PIMS Postdoctoral Fellowship.
FS is supported by an NSF grant and a Simons grant.
We thank the anonymous referees for their detailed feedback on an earlier version of our manuscript.

\section{Preliminaries}
\label{section: preliminaries}
Let $p>3$ be a prime and $K=\Q(\sqrt{d})$ an imaginary quadratic field.
Denote by $K_{\ac}$ the anti-cyclotomic $\Z_p$-extension of $K$.
For $n\geq 0$, the \emph{$n$-th layer} is the unique number field $K_n$ such that $K\subseteq K_n \subset K_{\ac}$ and $[K_n:K]=p^n$.
Note that $K_n$ is Galois over $\Q$ and its Galois group $\Gal(K_n/\Q)$ is (isomorphic to) the dihedral group of order $2p^n$.

Denote by $\Gamma$ the Galois group $\Gal(K_{\ac}/K)$ and pick a topological generator $\gamma\in \Gamma$.
The \emph{Iwasawa algebra} $\Lambda$ is the completed group algebra $\Z_p\llbracket \Gamma \rrbracket :=\varprojlim_n \Z_p[\Gamma/\Gamma^{p^n}]$.
Fix an isomorphism of rings $\Lambda\simeq\Z_p\llbracket T\rrbracket$ by sending $\gamma -1$ to the formal variable $T$.

\subsection{}
Let $\EC_{/\Q}$ be an elliptic curve with good reduction at $p$ and of conductor $N_\EC$ so that the discriminant of $K=\Q(\sqrt{d})$ is coprime to $pN_\EC$.
The main objects of study in this paper are the \emph{minimal Selmer groups} defined in \cite[Section~3.1]{PW11}.
For the convenience of the reader, we work with a less technical but equivalent definition of these $p$-primary Selmer groups than that of \cite{PW11}\footnote{For the equivalence of the definitions in the setting of this paper, we refer the reader to \cite[Appendix~A]{PW11}. The two sentence summary of their discussion is that Selmer groups are defined via a collection of local conditions, which the authors call \textit{Selmer structures}. While these Selmer structures may be different at $K_n$, they become the same as one takes the limit and works with $K_{ac}$.}.

Choose a finite set of primes $S$ containing the primes $v|p$ in $K$, the archimedean primes, and the primes at which $\EC$ has bad reduction.
For any finite extension $L/K$, write $S(L)$ to denote the set of primes $w$ of $L$ such that $w$ lies above a prime $v\in S$.

For ease of notation, define
\[
J_v(\EC/L) = \prod_{w|v} H^1\left(L_w, \EC[p^\infty]\right)/\left(\EC(L_w)\otimes \Qp/\Zp\right),
\]
where the product is over all primes $w$ of $L$ lying above $v$.
Following R.~Greenberg \cite[p.~107]{Gre89} (see also \cite[p.~20]{Gre98_PCMS}), the \emph{$p$-primary Selmer group over $L$} is defined as follows
\[
\Sel_{p^\infty}(\EC/L):=\ker\left\{ H^1\left(L,\EC[p^{\infty}]\right)\longrightarrow \bigoplus_{v} J_v(\EC/L)\right\}.
\]
It is also possible to define the \emph{$p$-primary Selmer group} by using a smaller Galois group,
\[
\Sel_{p^\infty}(\EC/L):=\ker\left\{ H^1\left(K_S/L,\EC[p^{\infty}]\right)\longrightarrow \bigoplus_{v\in S} J_v(\EC/L)\right\}.
\]
The fact that these two definitions agree follows from \cite[Proposition~6.5 or Corollary~6.6]{Mil_ADT}.
For a detailed discussion, we refer the reader to \cite[Section 1.7 (Cassels-Poitou-Tate sequence)]{CS00_GCEC}.
Next, set $J_v(\EC/K_{\ac})$ to be the direct limit
\[
J_v(\EC/K_{\ac}):=\varinjlim_L J_v(\EC/L),
\]
where $L$ ranges over all number fields contained in $K_{\ac}$.
Taking direct limits, the \emph{$p$-primary Selmer group over $K_{\ac}$} can be defined as follows
\[
\Selpinf:=\ker\left\{ H^1\left(K_S/K_{\ac},\EC[p^{\infty}]\right)\longrightarrow \bigoplus_{v\in S} J_v(\EC/K_{\ac})\right\}.
\]
As explained earlier, over $K_{\ac}$ as well, we have an equivalent definition of the Selmer group,
\[
\Selpinf:=\ker\left\{ H^1\left(K_{\ac},\EC[p^{\infty}]\right)\longrightarrow \bigoplus_{v} J_v(\EC/K_{\ac})\right\}.
\]
Note that the map above is a map of $\Lambda$-modules.
A $\Lambda$-module $M$ is said to be \emph{cofinitely generated} (resp. \emph{cotorsion}) if its Pontryagin dual $M^{\vee}:=\Hom_{\Z_p}\left(M, \Q_p/\Z_p\right)$ is finitely generated (resp. torsion) as a $\Lambda$-module.
A standard application of Nakayama's lemma shows that $\Selpinf$ is cofinitely generated.
However, this Selmer group \emph{need not be} $\Lambda$-cotorsion.

\subsection{}
Let $\EC_{/\Q}$ be an elliptic curve with supersingular reduction at $p> 3$.
Set $\widehat{\EC}$ to be the formal group of $\EC$ over $\Z_p$.
Let $L$ be a finite extension of $\Q_p$ with valuation ring $\mathcal{O}_L$ and let $\widehat{\EC}(L)$ denote $\widehat{\EC}(\mathfrak{m}_L)$, where $\mathfrak{m}_L$ is the maximal ideal in $L$.
Let $v$ be a prime above $p$ in $K_n$.
Following S.~Kobayashi \cite{Kob03}, we define the plus (and minus) norm groups as follows
\begin{align*}
\widehat{\EC}^+(K_{n,v})& :=
\left\{P\in \widehat{\EC}(K_{n,v}) \mid \tr_{n/m+1} (P)\in \widehat{\EC}(K_{m,v}), \text{ for }0\leq m < n\text{ and }m \text{ even}\right\},\\
\widehat{\EC}^-(K_{n,v})&:=
\left\{P\in \widehat{\EC}(K_{n,v}) \mid \tr_{n/m+1} (P)\in \widehat{\EC}(K_{m,v}),\text{ for }0\leq m < n\text{ and }m 
\text{ odd}\right\},
\end{align*}
where $\tr_{n/m+1}:\widehat{\EC}(K_{n,v})\rightarrow \widehat{\EC}(K_{m+1,v})$ is the trace map with respect to the formal group law on $\widehat{\EC}$.
Define the \emph{plus (resp. minus) Selmer group} at the $n$-th layer of the $\Z_p$-extension as follows
\[
0 \rightarrow \Sel_{p^\infty}^{\pm}(\EC/K_n) \rightarrow \Sel_{p^\infty}(\EC/K_n) \rightarrow \prod_{v|p} \frac{H^1\left(K_{n,v},\EC[p^{\infty}]\right)}{\widehat{\EC}^{\pm}(K_{n,v})\otimes \Q_p/\Z_p}
\]
in view of $\left(\widehat{\EC}(L_v)^\pm\otimes \Qp/\Zp\right)\subset \left(\EC(L_v)\otimes \Qp/\Zp\right)$.
The plus (resp. minus) Selmer groups over $K_{\ac}$ are defined by taking direct limits, i.e.,
\[
\Sel_{p^\infty}^{\pm}(\EC/K_{\ac}) := \varinjlim_n \Sel_{p^\infty}^{\pm}(\EC/K_n).
\]

\subsection{}
Let $M$ be a finitely generated $\Lambda$-module and $M^\vee$ denote its Pontryagin dual.
The \emph{Structure Theorem for $\Lambda$-modules} (see \cite[Theorem~13.12]{washington1997}) asserts that $M$ is pseudo-isomorphic to a finite direct sum of cyclic $\Lambda$-modules, i.e., there is a map of $\Lambda$-modules
\[
M\longrightarrow \Lambda^r\oplus \left(\bigoplus_{i=1}^s \Lambda/(p^{\mu_i})\right)\oplus \left(\bigoplus_{j=1}^t \Lambda/(f_j(T)) \right)
\]
with finite kernel and cokernel.
Here, $\mu_i>0$, and $f_j(T)$ are distinguished polynomials (monic polynomials with non-leading coefficients divisible by $p$).
The $\mu$-invariant of $M$ is defined as the power of $p$ in $ f_{M}(T) := p^{\sum_{i} \mu_i} \prod_j f_j(T)$.
More precisely,
\[
\mu_p(M):=\begin{cases}
\sum_{i=1}^s \mu_i & \textrm{ if } s>0\\
0 & \textrm{ if } s=0.
\end{cases}
\]
\begin{remark}
We are interested in the $\Lambda$-modules $\Sel_{p^\infty}(\EC/K_{\ac})$ (or $\Sel^{\pm}_{p^\infty}(\EC/K_{\ac})$) when they are $\Lambda$-cotorsion, see \cite[Theorem~1.3]{PW11}.
We write $\mu(\EC/K_{\ac})$ (or $\mu^{\pm}(\EC/K_{\ac})$) for the $\mu$-invariant of the appropriate Selmer group.
\end{remark}

\subsection{}
\label{introduce hypo CR}
Keeping the notation introduced earlier, write $N_\EC=N_\EC^+ N_\EC^-$ where $N_\EC^+$ is the product of the bad reduction primes that are split in $K$ and $N_\EC^-$ is the product of the bad reduction inert primes.
The following hypothesis guarantees the cotorsionness of the above remark in a large number of cases.
It was introduced in \cite{PW11} (see \cite[Assumption~1.1 and Remark~1.4]{KPW17} for correction).

\begin{labelledconj}{Hypothesis choired}{hyp:CR}
For a prime $p> 3$, this hypothesis is the following list of conditions:
\begin{enumerate}[\textup{(}1\textup{)}]
\item $\barrho_{\EC,p}:\Gal(\widebar{\Q}/\Q)\rightarrow \GL_2(\F_p)$ is surjective.
\item If $q$ is a prime with $q|N_\EC^-$ and $q\equiv \pm{1} \pmod{p}$, then $\barrho_{\EC,p}$ is ramified at $q$.
\item $N_\EC^-$ is square-free and the number of primes dividing $N_\EC^-$ is odd.
\item $a_p\not\equiv \pm 1\pmod{p}$.
\end{enumerate}
\end{labelledconj}

\begin{remark}In \cite{KPW17} and \cite{PW11}, various subsets of the above hypotheses are named `CR' -- CR stands for `controlled ramification' \cite{chanho}.
We make a few clarifying remarks.
In \cite{PW11}, Condition~(4) had been omitted.
Also, the results in \emph{loc. cit.} implicitly assumed that $\barrho_{\EC,p}$ is ramified at all primes dividing $N_\EC^+$.
This latter assumption of $N^+$-minimality can now be removed by a level-lowering trick introduced in \cite{KPW17}, see Section~1.2 in \emph{loc. cit.} for details.
\end{remark}

Finally, we recall the main theorems  of interest in \cite{PW11, KPW17}.

\begin{theorem}[{\cite[Theorem 1.3]{PW11} and \cite[Theorem 2.2]{KPW17}}]
Let $\EC$ be an elliptic curve with good reduction at a prime $p>3$ and conductor $N_{\EC}$ so that  \ref{hyp:CR} holds. 
\begin{enumerate}[\textup{(}1\textup{)}]
\item When $a_p=0$, let $K$ be an imaginary quadratic field so that the discriminant of $K$ is coprime to $pN_{\EC}$ and $p$ splits in $K$ into two primes that are totally ramified in $K_{\ac}/K$. Then $\Sel_{p^\infty}^{\pm}(\EC/K_{\ac})^{\vee}$ is $\Lambda$-torsion with $\mu^\pm(\EC/K_{ac})=0$.
\item When $a_p\not\in\{\pm 1, 0\}$, let $K$ be an imaginary quadratic field so that the discriminant of $K$ is coprime to $pN_{\EC}$. Then $\Sel_{p^\infty}(\EC/K_{\ac})^{\vee}$ is $\Lambda$-torsion with $\mu(\EC/K_{ac})=0$.
\end{enumerate}

\end{theorem}

\section{Average results: Varying over primes} 
\label{varying over primes}
In this short warm-up section, we study what happens when one fixes a pair $(\EC_{/\Q}, K)$ and varies $p$.
The main theorem shows that for almost all ordinary primes\footnote{and hence almost all primes}, the anti-cyclotomic Selmer groups are cotorsion with trivial $\mu$-invariant.  
We do this by proving that \ref{hyp:CR} holds for appropriately large ordinary primes, so that the aforementioned work of Pollack--Weston guarantees the desired $\Lambda$-torsionness and also guarantees the vanishing of the $\mu$-invariant(s).
Note that we require \emph{no} hypothesis on the Mordell--Weil rank of the elliptic curve over $\Q$ (or $K$).

The results of Pollack--Weston require that $K$ is an imaginary quadratic field of discriminant coprime to $pN_\EC$.
In the supersingular reduction case, it is also required that $p$ splits in $K$ and that the primes above $p$ are totally ramified in $K_{\ac}/K$.
\begin{remark}
\label{remark: brink ramified}
If $p$ does not divide the class number of $K$, denoted by $h_K$, then the primes above $p$ in $K$ are totally ramified in the anti-cyclotomic $\Zp$-extension, see \cite[p.~2131 last paragraph]{Bri07}.
\end{remark}

\begin{theorem}
\label{Th: CR-vary p}
Fix a pair $(\EC_{/\Q}, K)$, where $K$ is an imaginary quadratic field as above and $\EC_{/\Q}$ is an elliptic curve without complex multiplication so that $N_\EC^-$ is a product of an odd number of distinct primes.
Let $p>68 N_\EC(1+\log\log N_\EC)^{1/2}$ be a prime of good reduction.
\begin{enumerate}[\textup{(}1\textup{)}]
\item If $a_p=0$, $p$ splits in $K$ and $p\nmid h_K$, then the $p$-primary signed Selmer groups $\Sel^{\pm}_{p^\infty}(\EC/K_{\ac})$ are $\Lambda$-cotorsion with $\mu^{\pm}(\EC/K_{\ac})=0$.
\item If $a_p\not\equiv 0, \pm 1\pmod{p}$ and $p$ is unramified in $K$, then the $p$-primary Selmer group $\Sel_{p^\infty}(\EC/K_{\ac})$ is $\Lambda$-cotorsion with $\mu(\EC/K_{\ac})=0$.
\end{enumerate}
\end{theorem}

\begin{proof}
We check for which primes the four criteria appearing in \ref{hyp:CR} hold.
The last two always hold by assumption.
The first two are satisfied for $ p >68 N_\EC(1+\log\log N_\EC)^{1/2}$:
\begin{enumerate}[\textup{(}i\textup{)}]
\item $\overline{\rho}_{\EC,p}$ is surjective:\\
Serre's Open Image Theorem (see for example \cite{Serre72}) asserts that for a fixed elliptic curve $\EC$ without complex multiplication there exists a positive constant $C_\EC$ such that for $p>C_\EC$ the mod-$p$ representation is surjective.
The bound $C_\EC\leq 68 N_\EC(1+\log\log N_\EC)^{1/2}$ is due to A.~Kraus \cite{kraus}.
Hence, this condition is satisfied as soon as $p>68 N_\EC(1+\log\log N_\EC)^{1/2}$.
\item If $q|N_\EC^{-}$ and $q\equiv \pm 1\pmod{p}$, then $\overline{\rho}_{\EC,p}$ is ramified at $q$:\\
For any of the finitely many primes $q$ dividing $N_\EC^-$, the condition $q\equiv \pm 1 \pmod{p}$ is never satisfied for $p\gg 0$, e.g. $p>68N^-_\EC$.
In particular, this condition is vacuous for all good reduction primes $p>68 N(1+\log\log N)^{1/2}$.
\end{enumerate}
We thus conclude that \ref{hyp:CR} is satisfied for all sufficiently large $p$ in either case.

To complete the proof we apply \cite[Theorems~1.1 and 1.3]{PW11}.
Note that in (1), Remark~\ref{remark: brink ramified} ensures that $p$ is totally ramified in $K_{\ac}/K$.
This allows using the said results from \emph{loc. cit}.
\end{proof}

\begin{remark}
\begin{enumerate}[\textup{(}1\textup{)}]
\item A.~Cojocaru has obtained bounds similar to that of Kraus in \cite[Theorem~2]{Coj05}.
Thus, our theorem may be improved by replacing Kraus's bound $68 N_\EC(1+\log\log N_\EC)^{1/2}$ by Cojocaru's $\frac{4\sqrt{6}}{3}N_\EC\displaystyle{\prod_{p|N_\EC}\left(1+\frac{1}{p}\right)^\frac{1}{2}}$ whenever it is smaller.
\item It is conjectured that $C_\EC=37$, see \cite[p.~399]{serrecebo}.
Consequently, the Selmer group $\Sel_{p^\infty}^{\pm}(\EC/K_{\ac})$ (resp. $\Sel_{p^\infty}(\EC/K_{\ac})$) would be $\Lambda$-cotorsion as soon as $p>\max\{N_\EC^-+1,37, h_K\}$ (resp. $p>\max\{N_\EC^-+1,37\}$) and $p$ splits (resp. is unramified) in $K$.
\item Elkies proved that given $\EC_{/\Q}$, there are infinitely many primes at which it has supersingular reduction \cite[Theorem~1]{Elk87}.
By the Chebotarev density theorem, half of the primes split in $K$.
A priori it is not obvious that given a pair $(\EC_{/\Q}, K)$ there are infinitely many primes of supersingular reduction of $\EC$ which split in $K$.
It is possible to find non-CM elliptic curves over $\Q$ for which only finitely many supersingular primes split in a given \emph{imaginary quadratic} field.
For example (see \cite[p.~1]{Wal10}) the supersingular primes of \href{https://www.lmfdb.org/EllipticCurve/Q/15/a/7}{$X_1(15)$}, given by the equation
\[
Y^2 + XY + Y = X^3 + X^2,
\]
satisfy the property that $p\equiv 3 \pmod{4}$\footnote{Note that not all primes of the form $p\equiv 3 \pmod{4}$ are supersingular.}.
However, such primes do not split in $K=\Q(i)$.
In \cite[Section~4.1]{Wal10}, it is explained that in any real (resp. imaginary) quadratic field $K$, there is a bias \emph{in favour of} (resp. \emph{against}) the occurrence of supersingular primes that split in the field.
However, the averaging results in \emph{loc. cit.} suggest that if $\EC$ is a non-CM elliptic curve picked \emph{at random} then there is a positive proportion of supersingular primes which split in $K$.
\end{enumerate}
\end{remark}

\section{Average results: Varying the imaginary quadratic field}
\label{section: Varying K}
Fix an elliptic curve $\EC_{/\Q}$ without complex multiplication of square-free conductor $N_\EC$ and a prime $p> 3$ of good reduction with $a_p\not\equiv \pm{1}\pmod{p}$.
In this section, we count for what proportion of imaginary quadratic fields the associated Selmer groups are $\Lambda$-cotorsion with $\mu$-invariant equal to 0.
We prove the theorem separately when $p$ is a prime of supersingular or ordinary reduction.

\begin{remark}
	The methods in this section have been written so that they can be extended to the case of weight 2 modular forms. 
\end{remark}


\subsection{The supersingular case}
In this section, we assume $a_p=0$ (which is equivalent to $p$ being supersingular, since $p>3$.)
Varying over \emph{imaginary quadratic} fields $\Q(\sqrt{d})$ we estimate how often $\Sel_{p^\infty}^{\pm}(\EC/\Q(\sqrt{d})_{\ac})$ is $\Lambda$-cotorsion with $\mu^{\pm}(\EC/\Q(\sqrt{d})_{\ac})=0$.
To this end, we estimate for what proportion of \emph{imaginary quadratic} fields the following properties hold:
\begin{enumerate}[\textup{(}1\textup{)}]
\item $\gcd\left(pN_\EC, \abs{\Dd}\right)=1$,
\item \ref{hyp:CR} is satisfied by the triple $(\EC_{/\Q}, \Q(\sqrt{d})_{\ac}, p)$, \emph{and}
\item $p$ splits in $\Q(\sqrt{d})$.
\item $p$ does not divide the class number of $\Q(\sqrt{d})$.
\end{enumerate}
As for the last property, the \emph{Cohen--Lenstra heuristics} predict that among all imaginary quadratic fields, the proportion for which $p$ \emph{divides} the class number is \cite{bhandmurty}, \cite[Section 9.I]{CohenLenstra}
\begin{equation}
\label{Cohen Lenstra imaginary quadratic}
c_p = \frac{6}{\pi^2}\left( 1 - \prod_{j=1}^\infty \left(1 -\frac{1}{p^j} \right)\right).
\end{equation}
We remind the readers that this proportion is expected to be \emph{positive}.
A result of K.~Horie and Y.~Onishi (see \cite{HO88}) establishes that there are infinitely many imaginary quadratic number fields such that $p$ \emph{does not} divide the class number.
In \cite{KO99}, W.~Kohnen and K.~Ono have obtained lower bound asymptotic but we are still quite far from establishing the Cohen--Lenstra heuristics.


\begin{defi}
Let $\mathcal{S}$ be a subset of imaginary quadratic fields.
Define the \emph{density of $\mathcal{S}$} as
\[
\delta(\mathcal{S}):=\lim_{x\rightarrow \infty} \frac{\#\left\{\Q(\sqrt{d}): \abs{\Dd}<x \textrm{ and } \Q(\sqrt{d})\in \mathcal{S}\right\}}{\#\left\{\Q(\sqrt{d}): d<0, \ \abs{\Dd}<x\right\}}.
\]
\end{defi}

\begin{defi}
\label{ss choired defi}
Given a pair $(\EC_{/\Q},p)$ of an elliptic curve $\EC$ of conductor $N_\EC$ and a prime $p$ of good reduction, define $Q^-(\textrm{choired}, p+)$ as the following set
\small{\[
\left\{\Q(\sqrt{d}):d<0, \ \gcd\left(\abs{\Dd},pN_\EC\right)=1, \ p \textrm{ splits in }\Q(\sqrt{d}), \textrm{ \ref{hyp:CR} holds for } (\EC_{/\Q}, \Q(\sqrt{d}), p)\right\}.
\]	}
\end{defi}	

\begin{defi}
\label{kountertoCR}
Given an elliptic curve $\EC_{/\Q}$ of conductor $N_\EC$ and a prime $p$, define $k$ to be the number of bad primes $q| N_\EC$ that satisfy both of the following:
\begin{enumerate}[\textup{(}1\textup{)}]
\item $q \equiv \pm 1\pmod{p}$,
\item $\bar{\rho}_{\EC,p}$ is unramified at $q$.
\end{enumerate}
\end{defi}

\begin{remark}
In words, $k$ counts the number of bad primes $q$ that would defy the key assumption of \cite{KPW17} if $q$ were inert in the quadratic imaginary field to be chosen.
Of course, since we are working under their assumption, the primes in Definition \ref{kountertoCR} must split -- they are `kounterexamples,' i.e. fake counterexamples to the key assumption of \cite{KPW17}, hence the choice of the letter $k$.
\end{remark}

\begin{theorem}
\label{thm: vary K hypothesis CR}
Fix a pair $(\EC_{/\Q},p)$ so that 
\begin{enumerate}[\textup{(}1\textup{)}]
\item $\EC_{/\Q}$ is an elliptic curve with square-free conductor $N_\EC = \prod_{i=1}^r q_i$, and
\item $p> 3$ is a prime at which $\EC$ has good supersingular reduction, $\overline{\rho}_{\EC,p}$ is surjective, and $k<r$.
\end{enumerate}	
Then
\[
\delta\left(Q^-(\textrm{\emph{choired, $p+$}})\right)=\frac{pN_\EC}{2^{k+2} (p+1) \prod_{q_i|N_\EC} (q_i + 1)}.
\]
Let $c^{*}_p$ denote the proportion of imaginary quadratic fields in $Q^-(\textrm{\emph{choired, $p+$}})$ with $p$ dividing the class number.
The proportion of imaginary quadratic fields with $\gcd\left(\abs{\Dd},pN_\EC\right) = 1$, the prime $p$ splits in $\Q(\sqrt{d})$, and $\Sel_{p^\infty}^{\pm}(\EC/\Q(\sqrt{d})_{\ac})$ is $\Lambda$-cotorsion with $\mu^{\pm}$-invariant equal to zero is \emph{at least}
\[
\frac{pN_\EC}{2^{k+2} (p+1) \prod_{q_i|N_\EC} (q_i + 1)}\cdot (1-c^*_p).
\]
\end{theorem}

First, we introduce some notation.
\begin{defi}
Define $\Pi_p(N_\EC)$ to be the set of prime divisors of $N_\EC$ together with $p$, i.e.,
\[
\Pi_p(N_\EC) = \{p, q_1, \ldots, q_r\}.
\]
Choose the indices so that $q_i$ with $i\leq k$ are the primes that satisfy the conditions of Definition~\ref{kountertoCR}.
This choice allows keeping track of primes that potentially violate condition \textup{(}2\textup{)} in \ref{hyp:CR}.
\end{defi}

\begin{defi}
\label{defn partition of Pi}
For any partition $\Pi=\Pi^-\sqcup \Pi^+$ of $\Pi_p(N_\EC)$ into two disjoint parts, define 
\begin{multline*}
Q^-(\Pi) := \left\{\Q(\sqrt{d}): d<0, \ \gcd\left( \abs{\Dd}, pN_\EC\right)=1, \textrm{ primes in }\Pi^- \textrm{ are inert in }\Q(\sqrt{d}),  \right.\\
\left. \textrm{ and primes in }\Pi^+ \textrm{ split in }\Q(\sqrt{d}) \right\}.
\end{multline*}
Denote by $\delta_{\Pi}$ the density of $Q^-(\Pi)$, i.e. $\delta_{\Pi}:=\delta\left(Q^-\left(\Pi\right)\right)$.
\end{defi}

\begin{lemma}
\label{residueclasses}
Pick a partition $\Pi=\Pi^-\sqcup \Pi^+$ of $\Pi_p(N_\EC)$ into two disjoint parts.
\begin{enumerate}
\item When $2\nmid N_\EC$, there exists a subset $\mathfrak{r}_{\Pi}$ of $(\Z/pN_\EC \Z)^*$ of size $\frac{\varphi(pN_\EC)}{2^{r+1}}$ so that 
\begin{align*}
\Q(\sqrt{d})\in Q^-(\Pi)\iff &\Dd \mod pN_\EC\in \mathfrak{r}_{\Pi}.
\end{align*}
\label{residueclasses2}
\item When $2\mid N_\EC$, there exists a subset $\mathfrak{r}_{\Pi}$ of $(\Z/4pN_\EC \Z)^*$ of size $\frac{\varphi(pN_\EC)}{2^{r}}$ so that 
\begin{align*}
\Q(\sqrt{d})\in Q^-(\Pi)\iff &\Dd \mod 4pN_\EC\in \mathfrak{r}_{\Pi}.
\end{align*}
\end{enumerate}
In either case, $Q^-(\Pi)$ corresponds to $\frac{1}{2^{r+1}}$ of the possible residue classes for discriminants of quadratic imaginary fields coprime to $pN_\EC$.
\end{lemma}

\begin{proof}
From \cite[Proposition 5.16]{cox}, we have that $\Q(\sqrt{d})\in Q^-(\Pi)$ if and only if for all $q\in \Pi^-$, the Kronecker symbol $\left(\frac{\Dd}{q}\right)=-1$ and for all $q\in \Pi^+$, the Kronecker symbol $\left(\frac{\Dd}{q}\right)=+1$.

We first handle the case $2\nmid N_\EC$.
For $q\in \Pi_p(N_\EC)$, denote by $\mathfrak{r}_{q}^+\subset (\Z/q\Z)^*$ the set of quadratic residues, and by $\mathfrak{r}_{q}^-$ the set of quadratic non-residues.
Note that for an odd prime $q$ we have that $\#\mathfrak{r}_{q}^+=\#\mathfrak{r}_q^-=\frac{q-1}{2}$, since exactly half the elements of $(\Z/q\Z)^*$ are squares.
Define
\begin{equation*}
\mathfrak{r}_{\Pi}:=\prod_{q\in\Pi^-}\mathfrak{r}_q^-\times \prod_{q\in\Pi^+}\mathfrak{r}_q^+\subset \prod_{q\in \Pi_p(N_\EC)} (\Z/q\Z)^* \cong (\Z/pN_\EC \Z)^*.
\end{equation*}

Thus, $\Q(\sqrt{d})\in Q^-(\Pi)$ if and only if $\Dd \mod pN_\EC \in \mathfrak{r}_{\Pi}$, as claimed.
It also follows that 
\[
\# \mathfrak{r}_{\Pi} = \frac{(p-1)}{2}\prod_{i=1}^r \frac{q_i-1}{2}=\frac{\varphi(pN_\EC)}{2^{r+1}}.
\]

To handle the case $2\mid N_\EC$, set
\[
\mathfrak{r}_2^+:=\{1\}\in\left(\Z/8\Z\right)^*
\]
Note that by definition of the Kronecker symbol, $ \Dd\mod 8\in\{1\}\subset\left(\Z/8\Z\right)^*$ if and only if $\left(\frac{\Dd}{2}\right)=1$.
We can then proceed analogously to the case when $N_\EC$ was odd, working with
\begin{equation*}
\mathfrak{r}_{\Pi}:=\prod_{q\in\Pi^-}\mathfrak{r}_{q}^-\times \prod_{q\in\Pi^+}\mathfrak{r}_q^+\subset \left(\Z/8\Z\right)^*\times \prod_{\text{odd }q\in \Pi_p(N_\EC)} (\Z/q\Z)^* \cong (\Z/4pN_\EC \Z)^*.
\end{equation*}

For the last assertion, note that when $N_\EC$ is odd, $\Dd$ can reduce to any element in $(\Z/pN_\EC\Z)^*$, while for $N_\EC$ even, the reduction in the $(\Z/8\Z)^*$ part lies in $\{1,5\}$ as $2$ does not ramify by assumption.
\end{proof}

\begin{proposition}
\label{lemma: count imaginary quadratic with disc coprime to pNE}
Let $M$ be any square-free integer.
Define
\[
Q^-\left(x,D\perp M\right):= \left\{ \Q(\sqrt{d}) \textrm{ imaginary} : \ \abs{\Dd}< x \text{ and }\gcd\left(\Dd, M\right)=1 \right\}.
\]
Then asymptotically, 
\[
\lim_{x\rightarrow \infty}\# Q^-\left(x,D\perp M\right)\sim \frac{1}{2}\frac{x}{\zeta(2)}\frac{M}{ \prod_{q|M}(q +1)}
\]
\end{proposition}

\begin{proof}
A result of Prachar \cite[formula~1]{Pra58} (see \cite{MSE_PH} for two proofs by Humphries) says that
\[
\lim_{x\rightarrow\infty}\# \left\{ n \textrm{ square-free} : \ 0< n< x, \ n\equiv a \pmod{b} ,\ \gcd(b,a)=1\right\} \sim \frac{x}{\zeta(2)}\frac{1}{b}\prod_{q\mid b}\left(1-\frac{1}{q^2}\right)^{-1},
\]
where $a$ and $b$ are integers, and the $q$'s are primes.

\item We first handle the case $2\nmid M$.\newline
Put $b=4M$.
Then there are $\displaystyle{\prod_{q|M}(q-1)}$ congruence classes $a\pmod b$ with $\gcd(a,4M)=1$ and $a\equiv 1 \mod 4$, over which we sum the above Prachar--Humphries estimate.
Therefore,
\begin{align*}
&\lim_{x\rightarrow\infty}\# \left\{ n \textrm{ square-free} : \ \ n\equiv 1\pmod{4}, \ 0<n< x, \ \gcd\left( n , M\right)=1 \right\} \\
 \tag{*}\label{star}\sim & \prod_{q|M}(q-1)\times \frac{x}{\zeta(2)}\frac{1}{4M}\prod_{q|4M}\left(1- \frac{1}{q^2}\right)^{-1}\\
= & \prod_{q|M}(q-1)\times \frac{x}{\zeta(2)}\frac{1}{4M}\frac{4}{3}\prod_{q|M} \frac{q^2}{q^2-1} \\ 
= &\frac{x}{\zeta(2)}\frac{1}{3}\frac{M}{ \prod_{q|M}(q+1)}.
\end{align*}
The same estimate works for the congruence class $n\equiv 3\pmod{4}$.
To handle the congruence class $n\equiv 2\pmod{4}$, note that $n\equiv 2\pmod{4}$ with $(n,M)=1$ is square-free if and only if $\frac{n}{2}$ is.
But $\frac{n}{2}$ can be $\equiv 1\pmod{4}$ or $\equiv 3\pmod{4}$, so by using the above argument, 
\begin{align*}
&\lim_{x\rightarrow\infty}\# \left\{ n \textrm{ square-free} : \ \ n\equiv 2\pmod{4}, \ 0<n< x, \ \gcd\left( n , M\right)=1 \right\} \\
= &\lim_{x\rightarrow\infty}\# \left\{ n \textrm{ square-free} : \ \ n\equiv 1\text{ or } 3\pmod{4}, \ 0<n< \frac{x}{2}, \ \gcd\left( n , M\right)=1 \right\}\\
\sim &2\times\frac{x/2}{\zeta(2)}\frac{1}{3}\frac{M}{ \prod_{q|M}(q+1)}.
\end{align*}

Since we are assuming that $2\nmid M$, note that $\gcd\left(\abs{\Dd}, M\right)=1 \iff \gcd(\abs{d}, M)=1$.
Hence, we are interested in the following estimate: 
\begin{align*}
& \tag{**}\label{star2}\lim_{x\rightarrow\infty}\# \left\{ d \textrm{ square-free} : \ d<0, \ \abs{\Dd}< x, \ \gcd\left( \abs{d} , M\right)=1 \right\}\\
= & \lim_{x\rightarrow\infty}\# \left\{ d \textrm{ square-free} : \ d<0, \ \abs{d}< x, \ \gcd\left( \abs{d} , M\right)=1, \ d\equiv 1\pmod{4} \right\} \cup \\ 
& \lim_{x\rightarrow\infty}\# \left\{ d \textrm{ square-free} : \ d<0, \ \abs{d}< \frac{x}{4}, \ \gcd\left( \abs{d} , M\right)=1, \ d\equiv 2,3 \pmod{4} \right\}\\
\sim & \left(\frac{1}{3}\frac{x}{\zeta(2)}+\frac{1}{3}\frac{x/4}{\zeta(2)}+\frac{1}{3}\frac{x/4}{\zeta(2)}\right)\left(\frac{M}{ \prod_{q|M}(q +1)}\right)=\frac{1}{2}\frac{x}{\zeta(2)}\frac{M}{ \prod_{q|M}(q +1)}.
\end{align*}

\item The case $2|M$ is a bit easier. \newline
The simplification appears when decomposing \eqref{star2}  into $\pmod{4}$ congruence classes.
Since $2\mid M$, note that $\gcd\left(\abs{\Dd}, M\right)=1 $ implies $ d \equiv 1 \pmod{4}$.
Thus, 
\begin{align*}
&\lim_{x\rightarrow\infty}\# \left\{ \Q(\sqrt{d}) : \ d<0, \ \abs{\Dd}< x, \ \gcd\left( \abs{\Dd} , M\right)=1 \right\}\\
= & \lim_{x\rightarrow\infty}\# \left\{ d \textrm{ square-free}: \ d\equiv 1\pmod{4}, \ -x < d<0, \ \gcd\left( \abs{d} , M\right)=1 \right\}\\
\sim & \frac{1}{2}\frac{x}{\zeta(2)}\left(\frac{M}{ \prod_{q|M}(q +1)}\right).
\end{align*}
The reason for the different result in the Prachar--Humphries estimate (a factor of $\frac{1}{2}$ instead of $\frac{1}{3}$) is the following.
The integer $b=2M$ is a multiple of $4$, so there are again $ \prod_{q|M}(q-1)$ congruence classes $a\pmod b$ with $(a,M)=1$ and $a\equiv 1 \pmod{4}$.

Summing the Prachar--Humphries estimate over these classes, 
\begin{enumerate}
\item the term $\frac{1}{4M}$ in the equation \eqref{star} is replaced by $\frac{1}{2M}$, and 
\item the term $\frac{4}{3}$ in the line below disappears, since in the product $\prod_{q|M} \frac{q^2}{q^2-1}$, one of the primes $q=2$ by assumption.
\end{enumerate}
This accounts for a total scaling by a factor of $2\times\frac{3}{4}$.

The same argument also applies when $a \equiv 1\pmod{4}$ is replaced by $a \equiv 3 \pmod{4}$.
\end{proof}

\begin{remark}
\label{rem to be used in proof of thm 4.5}
\begin{enumerate}
\item \label{from cohen} It is well known that (see for example, \cite[Corollary~1.3]{CDO02}\footnote{This count is slightly different from the one in \cite{Bhargava_IMRN} where all fields are weighted by $1/\# \Aut$.})
\[
\lim_{x\rightarrow\infty}\# \left\{ \Q(\sqrt{d}) \textrm{ imaginary quadratic} : \ \abs{\Dd}< x \right\}\sim \frac{1}{2}\frac{x}{\zeta(2)}.
\]
In \emph{loc. cit.}, one finds an extra factor of $\frac{1}{2^{r_2}}$, but $r_2=0$ since $\Q$ has signature (1,0).
\item Proposition~\ref{lemma: count imaginary quadratic with disc coprime to pNE} says that for each prime $q$ with respect to which the coprimality condition is imposed on the discriminant, the proportion of imaginary quadratic fields reduces by a factor of $\frac{q}{q+1}$.
\end{enumerate} 
\end{remark}

We are now in a position to prove Theorem~\ref{thm: vary K hypothesis CR}.
\begin{proof}[Proof of Theorem~\ref{thm: vary K hypothesis CR}]
For \ref{hyp:CR} to be satisfied, we need
\begin{itemize}
\item $N^-_{\EC}$ is a product of an odd number of primes \emph{and}
\item none of the primes $q_i$ for $i\leq k$ divide $N^-_{\EC}$.
\end{itemize}
 Let $\mathcal{N}$ be the collection of all candidate subsets of the prime divisors of $N_\EC^-$.
In other words, let $\mathcal N$ be the collection of all subsets $\Pi^-\subseteq \{q_{k+1}, \dots, q_r\} $ for which $\# \Pi^-$ is odd.
Each such $\Pi^-$ determines a partition $\Pi=\Pi^-\sqcup\Pi^+$ of $\Pi_p(N_\EC)$ so that $p\in\Pi^+$.
Then
\[
Q^-(\textrm{\emph{choired, $p+$}}) = \bigsqcup_{\Pi\text{ so that }\Pi^-\in \mathcal{N}} Q^-(\Pi).
\]
Lemma~\ref{residueclasses} asserts that $Q^-(\Pi)$ corresponds to $\frac{1}{2^{r+1}}$ of the possible residue classes of discriminants coprime to $pN_\EC$.
Setting $M=pN_\EC$, Proposition~\ref{lemma: count imaginary quadratic with disc coprime to pNE} and Remark~\ref{rem to be used in proof of thm 4.5}\eqref{from cohen} tell us that the proportion of those imaginary quadratic fields with discriminant coprime to $pN_\EC$ among all quadratic imaginary fields is given by
\[
\mathfrak{d}:=\frac{pN_\EC}{(p+1)\prod_{q_i|N_\EC} (q_i + 1)},
\]
so that
\[
\delta_{\Pi}=\frac{1}{2^{r+1}}\mathfrak{d}=
\frac{pN_\EC}{2^{r+1}(p+1)\prod_{q_i|N_\EC} (q_i + 1)},
\]
Since $\# \mathcal{N}=2^{r-k-1}$ \cite[Exercise 1.1.13]{soberon2013}, 
\[
\delta\left(Q^-(\textrm{\emph{choired, $p+$}})\right)=\#\mathcal{N}\cdot \delta_{\Pi}.
\]
This proves our first assertion.

For the final assertion regarding the $\Lambda$-cotorsionness and triviality of the $\mu$-invariants, we need the two primes above $p$ to be \emph{totally ramified} in $\Q(\sqrt{d})_{\ac}/\Q(\sqrt{d})$.
A sufficient condition for this is the (additional) hypothesis that $p$ does not divide the class number of $\Q(\sqrt{d})$, which gives the factor of $1-c_p^*$ in the final assertion.
\end{proof}

\begin{remark}
The relationship between $c_p$ and $c^{*}_p$ is not immediate to the authors, e.g. should they be equal? 
We think it would be worthwhile to investigate this question in greater depth.
\end{remark}

\subsection{The ordinary case}
We now prove an analogue of Theorem~\ref{thm: vary K hypothesis CR} when $(\EC_{/\Q},p)$ is a fixed pair as before but $p$ is a prime of good \emph{ordinary} reduction of $\EC$.
We begin with the following definition:
\begin{defi}
Given a pair $(\EC_{/\Q},p)$ of an elliptic curve $\EC$ of square-free conductor $N_\EC$ and a prime $p$ of good ordinary reduction, define $Q^-(\textrm{choired},p\pm)$ -- or less precisely $Q^-(\textrm{choired})$ -- as the following set
\[
\left\{\Q(\sqrt{d}):d<0, \ \gcd\left(\abs{\Dd},pN_\EC\right)=1, \textrm{ \ref{hyp:CR} holds for } (\EC_{/\Q}, \Q(\sqrt{d}), p), \right\}.
\]
\end{defi}
Unlike in the supersingular case, here we count imaginary quadratic fields where $p$ is either inert or $p$ splits.
This is because, as pointed out in Section~\ref{introduce hypo CR}, Pollack--Weston have shown that for any triple $(\EC_{/\Q},\Q(\sqrt{d}),p)$ with $\Q(\sqrt{d})\in Q^-(\textrm{choired},p\pm)$, the associated Selmer group is $\Lambda$-cotorsion with $\mu(\EC/\Q(\sqrt{d})_{\ac})=0$.
In particular, there are no restrictions imposed on the splitting of $p$ in the imaginary quadratic field or on the class number of the said field.
We write
\begin{equation}
\label{ord as disjoint union}
Q^-(\textrm{choired},p\pm) = Q^-(\textrm{choired}, p+) \sqcup Q^-(\textrm{choired}, p-).
\end{equation}
Here, $p+$ (resp. $p-$) indicates quadratic fields in $Q^-(\textrm{choired},p\pm)$ with the additional property that $p$ splits (resp. remains inert), cf. Definition \ref{ss choired defi}.

\begin{theorem}
\label{thm: vary K hypothesis CR ordinary}
Fix a pair $(\EC_{/\Q},p)$ so that 
\begin{enumerate}[\textup{(}1\textup{)}]
\item $\EC_{/\Q}$ is an elliptic curve with square-free conductor $N_\EC = \prod_{i=1}^r q_i$, and
\item $p> 3$ is a prime at which $\EC$ has good ordinary reduction, $a_p\not\equiv 1\pmod{p}$, $\overline{\rho}_{\EC,p}$ is surjective, and $k<r$.
\end{enumerate}	
Then the proportion of imaginary quadratic fields with $\gcd\left(\abs{\Dd},pN_\EC\right) = 1$ and $\Sel_{p^\infty}(\EC/\Q(\sqrt{d})_{\ac})$ $\Lambda$-cotorsion with $\mu$-invariant equal to zero is at least
\[
\delta\left(Q^-(\textrm{choired},p\pm) \right) = \frac{pN_\EC}{2^{k+1} (p+1) \prod_{q_i|N_\EC} (q_i + 1)}.
\]
\end{theorem}

\begin{proof}
For \ref{hyp:CR} to be satisfied, we need
\begin{itemize}
\item $N^-_{\EC}$ is a product of an odd number of primes \emph{and}
\item none of the primes $q_i$ for $i\leq k$ divide $N^-_{\EC}$.
\item $a_p\not\equiv \pm{1}\mod p$.
\end{itemize}
However, the third condition involving the Fourier coefficients does not depend on the imaginary quadratic field $\Q(\sqrt{d})$.
In other words, this condition does not alter the count of $Q^-(\textrm{choired}, p\pm)$.
Hence, for each subset the calculations go through verbatim from the previous section.

Let $\mathcal N$ be the collection of all $\Pi^-\subseteq \{q_{k+1}, \dots, q_r\} $ for which $\# \Pi^-$ is odd.
Analogously to the supersingular case, to each such $\Pi^-$ we associate a partition of $\Pi $ of $\Pi_p(N_{\EC})= \{p, q_1, \ldots, q_r\}$ by
\[
\Pi = \begin{cases}
\Pi^- \sqcup \Pi^+,&\textrm{letting $p\in\Pi^+$} \textrm{ if } p \textrm{ splits in }\Q(\sqrt{d})\\
\left(\Pi^- \sqcup \{p\}\right) \sqcup \Pi^+, &\textrm{letting $p\not\in\Pi^+$}\textrm{ if } p \textrm{ is inert in }\Q(\sqrt{d}).\\
\end{cases}
\]
The proof then proceeds analogously to the proof of Theorem~\ref{thm: vary K hypothesis CR}, noting that adding $p$ to the collection of prescribed inert primes doesn't change the argument.
Thus,
\[
\delta\left(Q^-(\textrm{choired, $p+$}) \right) = \delta\left(Q^-(\textrm{choired, $p-$}) \right) = \frac{pN_\EC}{2^{k+2} (p+1) \prod_{q_i|N_\EC} (q_i + 1)}.
\]
This completes the proof of the theorem.
\end{proof}

\bibliographystyle{amsalpha}
\bibliography{references}
\end{document}